\documentclass[12pt,reqno]{amsart}

\newcommand\version{December 30, 2015}

\usepackage{amsmath,amsfonts,amsthm,amssymb,amsxtra}

\newtheorem{theorem}{Theorem}%[section]
\newtheorem{proposition}[theorem]{Proposition}
\newtheorem{lemma}[theorem]{Lemma}
\newtheorem{corollary}[theorem]{Corollary}

\theoremstyle{definition}

\theoremstyle{remark}

\newtheorem{remark}[theorem]{Remark}

\newcommand{\C}{\mathbb{C}}

\renewcommand{\epsilon}{\varepsilon}

\renewcommand{\phi}{\varphi}
\newcommand{\R}{\mathbb{R}}

\DeclareMathOperator{\codim}{codim}
\DeclareMathOperator{\dist}{dist}

\DeclareMathOperator{\im}{Im}
\DeclareMathOperator{\ran}{ran}
\DeclareMathOperator{\re}{Re}

\begin{document}

\title[Complex surface potentials --- \version]{Eigenvalues of Schr\"odinger operators with complex surface potentials}

\author{Rupert L. Frank}
\address{Rupert L. Frank, Mathematics 253-37, Caltech, Pasadena, CA 91125, USA}
\email{rlfrank@caltech.edu}

\begin{abstract}
We consider Schr\"odinger operators in $\R^d$ with complex potentials supported on a hyperplane and show that all eigenvalues lie in a disk in the complex plane with radius bounded in terms of the $L^p$ norm of the potential with $d-1<p\leq d$. We also prove bounds on sums of powers of eigenvalues.
\end{abstract}

\dedicatory{To Pavel Exner on the occasion of his 70th birthday}

\maketitle

\renewcommand{\thefootnote}{${}$} \footnotetext{\copyright\, 2015 by the author. This paper may be reproduced, in its entirety, for non-commercial purposes.}

\subsection*{Introduction and main results}

Recently there has been great interest in bounds on eigenvalues of Schr\"odinger operators with complex potentials. A conjecture of Laptev and Safronov \cite{LS} states that for a certain range of $p$'s, all eigenvalues of a Schr\"odinger operator lie in a disk in the complex plane whose radius is bounded from above in terms of only the $L^p$ norm of the potential. This conjecture was motivated by a corresponding result by Abramov, Aslanyan and Davies \cite{AAD} in one dimension and with $p=1$. In one part of the parameter regime the conjecture was proved in \cite{F}, and in the other part it was proved in \cite{FS} for radial potentials. For arbitrary potentials it is still open.

In this paper we deal with the analogue of this question for potentials supported on a hyperplane, which is a special case of what is called a `leaky graph Hamiltonian' in \cite{E}. More specifically, in $\R^d$, $d\geq 2$, we introduce coordinates $x=(x',x_d)$ with $x'\in\R^{d-1}$ and $x_d\in\R$ and consider the Schr\"odinger operator
\begin{equation}
\label{eq:ham}
-\Delta + \sigma(x')\delta(x_d)
\qquad\text{in}\ L^2(\R^d)
\end{equation}
with a complex function $\sigma$ on $\R^{d-1}$. If $\sigma\in L^p(\R^{d-1})$ for some $p>1$ in $d=2$ and $p\geq d-1$ in $d\geq 3$, this formal expression can be given meaning as an $m$-sectorial operator in $L^2(\R^d)$ through the quadratic form
\begin{equation}
\label{eq:quadform}
\int_{\R^d} |\nabla \psi(x)|^2\,dx + \int_{\R^{d-1}} \sigma(x') |\psi(x',0)|^2\,dx'
\end{equation}
with form domain $H^1(\R^d)$. If also $p<\infty$, it is a consequence of relative form compactness that the spectrum of this operator in $\C\setminus[0,\infty)$ consists of isolated eigenvalues of finite algebraic multiplicities; this is discussed below in more detail.

For \emph{real} $\sigma$, the variational principle for the lowest eigenvalue and the Sobolev trace theorem imply that any eigenvalue $E$ satisfies
$$
E \geq - \left( C_{\gamma,d} \int_{\R^{d-1}} \sigma(x')_-^{2\gamma+d-1}\,dx' \right)^{1/\gamma}
$$
for all $\gamma>0$ with a constant $C_{\gamma,d}$ independent of $\sigma$.

Our main result is an analogue of this bound for complex $\sigma$.

\begin{theorem}\label{main}
Let $d\geq 2$ and $0<\gamma\leq 1/2$. There is a constant $D_{\gamma,d}$ such that for any complex $\sigma\in L^{2\gamma+d-1}(\R^{d-1})$ and any eigenvalue $E\in\C$ of $-\Delta+\sigma(x')\delta(x_d)$ in $L^2(\R^d)$,
$$
|E|^\gamma \leq D_{\gamma,d} \int_{\R^{d-1}} |\sigma(x')|^{2\gamma+d-1}\,dx' \,.
$$
\end{theorem}

When $\gamma>1/2$ we cannot show that eigenvalues are bounded, but we can show that, if $(E_j)$ is a sequence of eigenvalues with $\re E_j\to \infty$, then $\im E_j\to 0$. The following theorem gives a quantitative version of this. We use the notation
\begin{equation}
\label{eq:deltaz}
\delta(z) := \dist(z,\C\setminus[0,\infty))=
\begin{cases}
|z| & \text{if}\ \re z\leq 0 \,,\\
|\im z| & \text{if}\ \re z> 0 \,.
\end{cases}
\end{equation}

\begin{theorem}\label{main.5}
Let $d\geq 2$ and $\gamma> 1/2$. There is a constant $D_{\gamma,d}$ such that for any complex $\sigma\in L^{2\gamma+d-1}(\R^{d-1})$ and any eigenvalue $E\in\C$ of $-\Delta+\sigma(x')\delta(x_d)$ in $L^2(\R^d)$,
$$
|E|^{1/2}\ \delta(E)^{(2\gamma-1)/2} \leq D_{\gamma,d} \int_{\R^{d-1}} |\sigma(x')|^{2\gamma+d-1}\,dx' \,.
$$
\end{theorem}

In dimensions $d\geq 3$ we also obtain a criterion for the absence of eigenvalues.

\begin{theorem}\label{main1}
Let $d\geq 3$. There is a constant $D_{0,d}$ such that for any complex $\sigma\in L^{d-1}(\R^{d-1})$, if
$$
\int_{\R^{d-1}} |\sigma(x')|^{d-1}\,dx' < D_{0,d}^{-1} \,,
$$
then $-\Delta+\sigma(x')\delta(x_d)$ in $L^2(\R^d)$ has no eigenvalue.
\end{theorem}

These three theorems are the analogues of the results in \cite{F,Fr3} for Schr\"odinger operators with usual potentials and our proof will follow the strategy in those papers (which, in turn, was motivated by \cite{AAD}).

Our final result concerns bounds on sums of powers of eigenvalues of $-\Delta + \sigma(x')\delta(x_d)$, which are analogues of the Lieb--Thirring inequalities \cite{LT}. Such bounds were shown in \cite{FL} for real $\sigma$ and, using the technique from \cite{FLLS}, extended to complex $\sigma$ provided one only considers eigenvalues outside of a cone around the positive real axis. The following theorem is useful for eigenvalues close to the positive real axis.

\begin{theorem}\label{main4}
Let $0<\gamma<1/2$ if $d=2$ and $0<\gamma\leq 1/2$ if $d\geq 3$. Let $\tau=0$ if $\gamma<(d-1)/(4d-6)$ and $\tau> ((4d-6)\gamma-(d-1))/(d-1-2\gamma)$ if $\gamma\geq (d-1)/(4d-6)$. Then there is a constant $L_{\gamma,d,\tau}$ such that, for any complex $\sigma\in L^{2\gamma+d-1}(\R^{d-1})$, the eigenvalues $(E_j)$ of $-\Delta+\sigma(x')\delta(x_d)$ in $L^2(\R^d)$, repeated according to their algebraic multiplicity, satisfy
$$
\left( \sum_j \delta(E_j) |E_j|^{-(1-\tau)/2} \right)^{2\gamma/(1+\tau)} \leq L_{\gamma,d,\tau} \int_{\R^{d-1}} |\sigma(x')|^{2\gamma+d-1}\,dx' \,.
$$ 
\end{theorem}

This is the analogue of a result from \cite{FSa} for Schr\"odinger operators with usual potentials. The method from \cite{Fr3} can probably be used to derive bounds for $\gamma> 1/2$, but to keep the exposition brief we do not pursue this here.

Our proof of Theorem \ref{main4} identifies, in the spirit of \cite{DeKa,BGK,DeHaKa0,DeHaKa,FSa,Fr3}, the eigenvalues of \eqref{eq:ham} with zeroes of an analytic function. As explained in detail in \cite{Fr3}, a result on zeroes of analytic functions \cite{BGK} plus inequalities on regularized determinants reduce the proof to resolvent bounds in trace ideals. These latter bounds are the content of Proposition \ref{unifsobti} and constitute the technical main result of this paper.

In conclusion we mention that there are two further methods which yield inequalities for sums of powers of eigenvalues. One method from \cite{DeHaKa0} relies on averaging the bounds from \cite{FL} with respect to the opening angle of the cone. Another method from \cite{Ha} is based on an extension of an inequality of Kato; see also \cite{Fr3}.

\begin{remark}\label{robin}
All the theorems reported here have an obvious analogue for the operator $-\Delta$ in $L^2(\R^d_+)$ with boundary condition $\partial\psi/\partial\nu = -\sigma\psi$. (Here $\R^d_+=\{x\in\R^d:\ x_d>0\}$ and $\partial/\partial\nu = - \partial/\partial x_d$.) This simply comes from the fact that the operator \eqref{eq:ham} leaves the spaces of functions which are even and odd with respect to $x_d$ invariant and on the former subspace it is unitarily equivalent to $-\Delta$ in $L^2(\R^d_+)$ with boundary condition $\partial\psi/\partial\nu = -(1/2)\sigma\psi$.
\end{remark}

%%%%%%%%%%%%%%%%%%%%%%%%%%%%%%%%%%%%%%%%%%%%%%%%%%%%%%%%

\subsection*{Uniform Sobolev inequalities}

In this section we prove a Sobolev inequality for functions on $\R^N$. (Later on, in the proof of Theorems \ref{main}, \ref{main.5} and \ref{main1} we will choose $N=d-1$.) The inequality involves the operator $\sqrt{-\Delta-z}$ and the crucial point is that the constant in the inequality depends only on $|z|$ but not on the argument of $z$. Such uniform Sobolev inequalities go back to Kenig, Ruiz and Sogge \cite{KRS} for $-\Delta-z$ and, in fact, our theorem follows by modifying their proof.

Since it comes be at no extra effort, we deal with the operators $(-\Delta-z)^s$ for arbitrary $0<s\leq(N+1)/2$. This operator acts as multiplication by $(\xi^2-z)^s$ in Fourier space. We will assume that $z\in\C\setminus[0,\infty)$, so $\xi^2-z\in\C\setminus(-\infty,0]$ for all $\xi\in\R^d$ and we can define $(\xi^2-z)^s = \exp(s\log(\xi^2-z))$ with the principal branch of the logarithm on $\C\setminus (-\infty,0]$.

\begin{proposition}\label{unifsob}
Let $0<s\leq (N+1)/2$ and assume that
$$
\begin{cases}
2N/(N+2s)\leq p \leq 2(N+1)/(N+1+2s) & \text{if}\ s<N/2 \,,\\
1 < p \leq 2(N+1)/(N+1+2s) & \text{if}\ s= N/2 \,,\\
1 \leq p \leq 2(N+1)/(N+1+2s) & \text{if}\ s>N/2 \,.
\end{cases}
$$
Then there is a constant $C_{N,p,s}$ such that for all $u\in W^{2s,p}(\R^N)$ and $z\in\C\setminus[0,\infty)$,
\begin{equation}
\label{eq:unifsob}
\|u\|_{p'} \leq C_{N,p,s}\, |z|^{-(Np+2ps - 2N)/(2p)} \left\|\left(-\Delta-z\right)^s u \right\|_p \,.
\end{equation}
Moreover, if $2(N+1)/(N+1+2s)<p\leq 2$, there is a constant $C_{N,p,s}$ such that for all $u\in W^{2s,p}(\R^N)$ and $z\in\C\setminus[0,\infty)$,
\begin{equation}
\label{eq:unifsob2}
\|u\|_{p'} \leq C_{N,p,s}\, \delta(z)^{-(Np+2ps-2N-2+p)/(2p)}\, |z|^{-(2-p)/(2p)} \left\|\left(-\Delta-z\right)^s u \right\|_p \,.
\end{equation}
\end{proposition}

We recall that $\delta(z)$ appearing in \eqref{eq:unifsob2} was defined in \eqref{eq:deltaz}. Moreover, $p'=p/(p-1)$.

\begin{proof}
Let $\zeta\in\C$ with $\re\zeta\geq 0$ and consider the operator
$$
T_\zeta(z):= e^{\zeta^2}\, \left( -\Delta - z \right)^{-\zeta} = e^{\zeta^2}\, e^{-\zeta \log(-\Delta-z)} \,,
$$
which is again defined as a multiplier in Fourier space with the same convention for the branch of the logarithm. Note that this is essentially the family of operators from \cite[Proof of Thm.\!\! 2.3]{KRS} with $\zeta=-\lambda$. Clearly, by bounding the multiplier in Fourier space, one finds
\begin{equation}
\label{eq:222}
\left\| T_\zeta(z) \right\|_{L^2\to L^2} \leq A
\qquad\text{if}\ \re \zeta =0
\end{equation}
with a constant $A$ depending only on $N$. Moreover, it is shown in \cite[Proof of Thm.~2.3]{KRS} that
\begin{equation}
\label{eq:12infty}
\left\| T_\zeta(z) \right\|_{L^1\to L^\infty} \leq B_{\re\zeta} |z|^{-(2\re\zeta-N)/2}
\qquad\text{if}\ N/2<\re \zeta\leq (N+1)/2
\end{equation}
with a constant $B_{\re\zeta}$ depending only on $N$ and $\re\zeta$. (Strictly speaking, this bound was only shown there with a constant independent of $z$ for $|z|\geq 1$, but the stated bound simply follows from this by scaling. Moreover, the assumption $N\geq 3$ in \cite{KRS} is irrelevant for the proof of \eqref{eq:12infty}.)

Since $T_s(z)$ coincides, up to a multiplicative constant, with the inverse of the operator $(-\Delta-z)^s$, if $s>N/2$ and $p=1$, we can choose $\zeta=s$ in \eqref{eq:12infty} and obtain the bound in the proposition.

For $p>1$ as in the theorem, with the extra assumption $p>2N/(N+2s)$ if $s<N/2$, we define $t:=sp/(2-p)$ and note that $N/2<t\leq (N+1)/2$ and $t>s$. Since the operators $T_\zeta(z)$ depend analytically on $\zeta$, we can use complex interpolation with the lines $\re\zeta=0$ and $\re\zeta=t$ and obtain
\begin{equation}
\label{eq:p2pprime}
\left\| T_{s}(z) \right\|_{L^p\to L^{p'}} \leq A^{2(p-1)/p} B_{ps/(2-p)}^{(2-p)/p} |z|^{-(Np+2ps-2N)/(2p)} \,,
\end{equation}
which again gives the claimed bounds.

For the first part of the proposition it remains to prove the bound for $p=2N/(N+2s)$ and $s<N/2$. Again as in \cite[Proof of Thm.\! 2.3]{KRS} we consider
$$
\tilde T_\zeta(z):= \frac{e^{\zeta^2}}{\Gamma((N-2\zeta)/2} \left( -\Delta - z \right)^{-\zeta} = \frac{e^{\zeta^2}}{\Gamma((N-2\zeta)/2} e^{-\zeta \log(-\Delta-z)} \,.
$$
Bound \eqref{eq:222} remains valid for $\tilde T_\zeta(z)$ and, as shown in \cite[Proof of Thm. 2.3]{KRS}, bound \eqref{eq:12infty} holds even for $\re\zeta=N/2=:t$. If $s<N/2$ one has $0<s<t$ and therefore one can again use complex interpolation to deduce an $L^{2N/(N+2s)}\to L^{2N/(N-2s)}$ bound for $\tilde T_s(z)$. This completes the proof of the proposition.

To prove the second part of the proposition we note that
$$
\|u\|_2 \leq \delta(z)^{-s} \left\| \left(-\Delta-z\right)^s u \right\|_2 \,.
$$
This, together with \eqref{eq:unifsob} for $p=2(N+1)/(N+1+2s)$, implies \eqref{eq:unifsob2} by standard (Riesz--Thorin) complex interpolation.
\end{proof}

%%%%%%%%%%%%%%%%%%%%%%%%%%%%%%%%%%%%%%%%%%%%%%%%%%%%%%%%%

\subsection*{Bound on the Birman--Schwinger operator}

Let us give the details of the definition of the operator \eqref{eq:ham} using the method from \cite[Sec.\! 4]{Fr3}.

We consider the operator $H_0:=-\Delta$ in the Hilbert space $\mathcal H :=L^2(\R^d)$ with form domain $H^1(\R^d)$. Moreover, let $\mathcal G:=L^2(\R^{d-1})$ and consider the operators $G$ and $G_0$ from $\mathcal H$ to $\mathcal G$ with domain $H^1(\R^d)$ defined by
$$
(G_0 \psi)(x') := \sqrt{\sigma(x')}\, \psi(x',0) \,,
\qquad
(G \psi)(x') := \sqrt{|\sigma(x')|}\, \psi(x',0) \,.
$$
(Here we write $\sqrt{\sigma(x')}=\sigma(x')/\sqrt{|\sigma(x')|}$ if $\sigma(x')\neq 0$ and $\sqrt{\sigma(x')}=0$ otherwise.) We claim that, if $\sigma\in L^p(\R^{d-1})$ with $1<p<\infty$ if $d=2$ and $d-1\leq p<\infty$ if $d\geq 3$, then
$$
G_0(H_0+1)^{-1/2}
\qquad\text{and}\qquad
G(H_0+1)^{-1/2}
\qquad\text{are compact}.
$$
When $\sigma$ is bounded and has support in a set of finite measure, this follows from the trace version of Rellich's compactness theorem, see, e.g., \cite[Thm.\!\! 6.3]{AF}. By the trace version of Sobolev's embedding theorem (see, e.g., \cite[Thm.\!\! 4.12]{AF}) and an argument as in \cite[Lem.\!\! 4.3]{Fr3} we obtain the assertion in the general case.

Thus, we have verified the assumptions of \cite[Lem.\!\! B.1]{Fr3} and we infer that the quadratic form \eqref{eq:quadform}, which is the same as $\|H_0^{1/2} \psi\|^2 + (G\psi,G_0\psi)$, is closed and sectorial and generates an $m$-sectorial operator $H$. Moreover, let $z\in\C\setminus[0,\infty)=\rho(H_0)$ and define the Birman--Schwinger operator
\begin{equation}
\label{eq:bs}
K(z) = G_0(H_0-z)^{-1}G^*
\qquad\text{in}\ L^2(\R^{d-1}) \,.
\end{equation}
Strictly speaking, since in our case the operators $G$ and $G_0$ are not closable, the operator $K(z)$ is defined as
$$
K(z) = \left( G_0 (H_0+1)^{-1/2} \right) (H_0+1)(H_0-z)^{-1}  \left( G (H_0+1)^{-1/2} \right)^* \,.
$$
The following \cite[Lem.\!\! B.1]{Fr3} is a version of the Birman--Schwinger principle.

\begin{lemma}
\label{bsprinc}
Let $z\in\C\setminus[0,\infty)$, then $1+K(z)$ is boundedly invertible if and only if $z\in\rho(H)$.
\end{lemma}

\begin{remark}
In passing we mention that \cite[Prop.\!\! B.2]{Fr3} yields that
\begin{align*}
& [0,\infty) = \left\{ z\in\C:\ \ran(H-z) \ \text{is not closed} \right\} \\
& \qquad\qquad\cup \left\{ z\in\C:\ \dim\ker(H-z)=\codim\ran(H-z) = \infty\right\}
\end{align*}
and
\begin{align*}
& \sigma(H)\setminus [0,\infty) \\
& = \left\{ z\in\C:\ \ran(H-z) \ \text{is closed and}\ 0<\dim\ker(H-z)=\codim\ran(H-z) <\infty\right\}.
\end{align*}
Moreover, the latter set is at most countable and consists of eigenvalues of finite algebraic multiplicities which are isolated in $\sigma(H)$. These facts, however, will not be relevant for the proof of Theorems \ref{main} and \ref{main1}.
\end{remark}

Our next goal is to find a convenient expression for the Birman--Schwinger operator. If we denote by $\Gamma$ the trace operator which restricts a function on $\R^d$ to $\R^{d-1}\times\{0\}$, then we have
$$
G_0 = \sqrt\sigma \,\Gamma \,,
\qquad
G = \sqrt{|\sigma|} \,\Gamma \,.
$$
Moreover, let us denote the Laplacian on $\R^{d-1}$ by $-\Delta'$.

\begin{lemma}\label{dn}
Let $z\in\C\setminus[0,\infty)$. Then
$$
\Gamma (-\Delta-z)^{-1}\Gamma^* = \frac12 \left( -\Delta'-z \right)^{-1/2} \,.
$$
\end{lemma}

\begin{proof}
Since $(-\Delta-z)^{-1}$ has integral kernel
$$
(2\pi)^{-d} \int_{\R^d} \frac{e^{i\xi\cdot(x-y)}}{\xi^2-z}\,d\xi \,,
\qquad x,y\in\R^d \,,
$$
the operator $\Gamma (-\Delta-z)^{-1}\Gamma^*$ has integral kernel
$$
(2\pi)^{-d} \int_{\R^{d-1}}\int_\R \frac{e^{i\xi'\cdot(x'-y')}}{(\xi')^2+\xi_d^2-z}\,d\xi'd\xi_d \,,
\qquad x',y'\in\R^{d-1} \,.
$$
The integral with respect to $\xi_d$ can be computed using
$$
\int_\R \frac{d\xi_d}{\xi_d^2 + b^2} = \frac{\pi}{b}
\qquad\text{if}\ \re b>0 \,.
$$
Thus, $\Gamma (-\Delta-z)^{-1}\Gamma^*$ has integral kernel
$$
\frac12 (2\pi)^{-d+1} \int_{\R^{d-1}}\frac{e^{i\xi'\cdot(x'-y')}}{\sqrt{(\xi')^2-z}}\,d\xi' \,,
\qquad x',y'\in\R^{d-1} \,,
$$
with the branch of the square root as described before Theorem \ref{unifsob}. This coincides with the integral kernel of the operator $(1/2)(-\Delta'-z)^{-1/2}$.
\end{proof}

\begin{remark}
One can also show that $\psi$ is an eigenfunction of \eqref{eq:ham} corresponding to an eigenvalue $E$ iff $\psi(x) = (\exp(-|x_d|\sqrt{-\Delta'-E})\phi)(x')$, where $\phi=\Gamma\psi$ satisfies $(\sqrt{-\Delta'-E}+\sigma/2)\phi=0$. This is closely related to the harmonic extension and the Dirichlet-to-Neumann operator for the Laplacian on $L^2(\R^d_+)$; see also Remark \ref{robin}. This observation was also crucial in \cite{FSh}.
\end{remark}

Combining Lemma \ref{dn} with Theorem \ref{unifsob} we obtain

\begin{corollary}\label{bs}
Let $0<\gamma\leq 1/2$ if $d=2$ and $0\leq\gamma\leq 1/2$ if $d\geq 3$. Then there is a constant $C_{\gamma,d}$ such that for all $\alpha_1,\alpha_2\in L^{2(2\gamma+d-1)}(\R^{d-1})$,
$$
\left\| \alpha_1 \Gamma (-\Delta-z)^{-1} \Gamma^* \alpha_2 \right\| \leq C_{\gamma,d} |z|^{-\frac\gamma{2\gamma+d-1}} \|\alpha_1\|_{2(2\gamma+d-1)} \|\alpha_2\|_{2(2\gamma+d-1)} \,.
$$
Moreover, if $\gamma>1/2$ there is a constant $C_{\gamma,d}$ such that for all $\alpha_1,\alpha_2\in L^{2(2\gamma+d-1)}(\R^{d-1})$,
$$
\left\| \alpha_1 \Gamma (-\Delta-z)^{-1} \Gamma^* \alpha_2 \right\| \leq C_{\gamma,d} \delta(z)^{-\frac{2\gamma-1}{2(2\gamma+d-1)}} |z|^{-\frac 1{2(2\gamma+d-1)}} \|\alpha_1\|_{2(2\gamma+d-1)} \|\alpha_2\|_{2(2\gamma+d-1)} \,.
$$
\end{corollary}

\begin{proof}
According to Lemma \ref{dn},
$$
\alpha_1 \Gamma (-\Delta-z)^{-1} \Gamma^* \alpha_2
= \frac 12 \ \alpha_1 \left(-\Delta'-z\right)^{-1/2} \alpha_2 \,,
$$
and, for any $1\leq p\leq\infty$,
$$
\left\| \alpha_1 \left(-\Delta'-z\right)^{-1/2} \alpha_2 \right\|
\leq \| \alpha_1 \|_{L^{p'}\to L^2}
\left\| \left(-\Delta'-z\right)^{-1/2} \right\|_{L^p\to L^{p'}} \|\alpha_2\|_{L^2\to L^p} \,.
$$
By H\"older's inequality, if $1\leq p\leq 2$,
$$
\|\alpha_1\|_{L^{p'}\to L^2} = \|\alpha_1\|_{2p/(2-p)}
\qquad\text{and}\qquad
\|\alpha_2\|_{L^2\to L^p} = \|\alpha_2\|_{2p/(2-p)} \,.
$$
We bound the norm of $(-\Delta'-z)^{-1}$ from $L^p$ to $L^{p'}$ by Theorem \ref{unifsob} with $N=d-1$. Bound \eqref{eq:unifsob} holds if $1<p\leq 4/3$ for $d=2$ and if $2(d-1)/d\leq p \leq 2d/(d+1)$ if $d\geq 3$. These conditions correspond precisely to our assumptions on $\gamma$ in the first part if we pick $p$ such that $2p/(2-p) = 2(2\gamma+d-1)$. Similarly, bound \eqref{eq:unifsob2} holds if $p>2d/(d+1)$ which corresponds to $\gamma>1/2$.
\end{proof}

%%%%%%%%%%%%%%%%%%%%%%%%%%%%%%%%%%%%%%%%%%%%%%%%%%%%%%%%%%

\subsection*{Proof of Theorems \ref{main}, \ref{main.5} and \ref{main1}}

Let $E$ be an eigenvalue of the operator \eqref{eq:ham}. We begin with the case $E\in\C\setminus[0,\infty)$, where we use the argument of \cite{AAD}; see also \cite{F}. Then, by the Birman--Schwinger principle (Lemma \ref{bsprinc}), $1+K(E) = 1+ \sqrt{\sigma}\Gamma(-\Delta-E)^{-1}\Gamma^* \sqrt{|\sigma|}$ is not boundedly invertible and therefore $\|K(E)\|\geq 1$. Combining this with the upper bound on $\|K(E)\|$ in the first part of Corollary \ref{bs} (with $\alpha_1=\sqrt\sigma$ and $\alpha_2=\sqrt{|\sigma|}$), we obtain
$$
1\leq C_{\gamma,d} |E|^{-\gamma/(2\gamma+d-1)} \|\sigma\|_{2\gamma+d-1} \,.
$$
This is the claimed bound on $|E|^\gamma$ for $0<\gamma\leq 1/2$ in Theorem \ref{main} and the condition for the absence of eigenvalues for $\gamma=0$ in Theorem \ref{main1}. Using the second part of Corollary~\ref{bs} instead, we obtain the claimed bound on $|E|^{1/2} \delta(E)^{(2\gamma-1)/2}$ for $\gamma>1/2$ in Theorem \ref{main.5}.

Now let $E\in[0,\infty)$ and denote by $\psi$ a corresponding eigenfunction. We use an approximation argument similar to \cite{FS}. For $\epsilon>0$ let
$$
\phi_\epsilon := G_0(-\Delta-E-i\epsilon)^{-1}(-\Delta-E)\psi \,,
$$
which is well-defined since $\psi\in H^1(\R^d)$. We claim that $\phi_\epsilon\to G_0\psi$ weakly in $L^2(\R^{d-1})$ as $\epsilon\to 0$. (Note that $G_0\psi\in L^2(\R^{d-1})$ is well-defined since $\psi\in H^1(\R^d)$.) In fact, by dominated convergence in Fourier space we conclude that for any $f\in L^2(\R^{d-1})$, as $\epsilon\to 0$,
\begin{align*}
(f,\phi_\epsilon) & = \left( \left( G_0 (-\Delta+1)^{-1/2} \right)^* f, (-\Delta-E-i\epsilon)^{-1}(-\Delta-E) (-\Delta+1)^{1/2}\psi \right) \\
& \to \left( \left( G_0 (-\Delta+1)^{-1/2} \right)^* f, (-\Delta+1)^{1/2}\psi \right) = (f,G_0\psi) \,.
\end{align*}

On the other hand, the eigenvalue equation for $\psi$ gives
$$
\phi_\epsilon = \left( G_0(-\Delta-E-i\epsilon)^{-1} G^* \right) \left( G_0\psi\right),
$$
and therefore, by Corollary \ref{bs},
$$
\|\phi_\epsilon\| \leq C_{\gamma,d} \left( E^2 + \epsilon^2 \right)^{-\gamma/(2(2\gamma+d-1))} \|\sigma\|_{2\gamma+d-1} \|G_0\psi\| \,.
$$
By weak semi-continuity of the norm we conclude that
$$
\|G_0\psi \| \leq \liminf_{\epsilon\to 0} \|\phi_\epsilon\| \leq \liminf_{\epsilon\to 0} C_{\gamma,d} \left( E^2 + \epsilon^2 \right)^{-\gamma/(2(2\gamma+d-1))} \|\sigma\|_{2\gamma+d-1} \|G_0\psi\| \,.
$$
Since $G_0\psi\not\equiv 0$ (otherwise $\psi$ would be an eigenfunction of $-\Delta$ with eigenvalue $E$), we finally obtain again
$$
1 \leq C_{\gamma,d} |E|^{-\gamma/(2\gamma+d-1)} \|\sigma\|_{2\gamma+d-1} \,,
$$
as claimed.

%%%%%%%%%%%%%%%%%%%%%%%%%%%%%%%%%%%%%%%%%%%%%%%%%%%%%%%%%%%%%%%%%%%%%%%%%

\subsection*{Uniform Sobolev inequalities in trace ideals}

By the argument in the proof of Corollary \ref{bs} we see that the uniform Sobolev inequality from Proposition \ref{unifsob} is equivalent to a bound of the operator norm of $\alpha_1 (-\Delta-z)^{-s} \alpha_2$ in terms of the $L^{2p/(2-p)}(\R^N)$-norms of $\alpha_1$ and $\alpha_2$ and an inverse power of $|z|$. In this section we improve this by showing that not only the operator norm, but also a trace ideal norm can be bounded in terms of the same quantities.

\begin{proposition}\label{unifsobti}
Let $0<s\leq (N+1)/2$ and assume that
$$
\begin{cases}
1\leq q \leq(N+1)/(2s) & \text{if}\ N<2s \,,\\
1< q\leq (N+1)/(2s) & \text{if}\ N=2s \,,\\
N/(2s)\leq q\leq (N+1)/(2s) & \text{if}\ N>2s \,.
\end{cases}
$$
In addition, if $N=1$ and $s\leq 1/2$ assume that $q<2$ and, if $N=2$ and $s\leq 1/2$ that $q>1/s$. Then there is a constant $C_{N,q,s}$ such that for all $\alpha_1,\alpha_2\in L^{2q}(\R^N)$ and $z\in\C\setminus[0,\infty)$,
$$
\left\| \alpha_1 \left(-\Delta-z\right)^{-s} \alpha_2 \right\|_r \leq C_{N,q,s} |z|^{-s+N/(2q)} \|\alpha_1\|_{2q} \|\alpha_2\|_{2q}
$$
with $r=2$ if $N=1$ and
$$
r= \max\left\{ (N-1)q/(N-qs), 2\right\} 
\qquad\text{if}\ N\geq 2 \,.
$$
\end{proposition}

For $s=1$ and $q\leq (N+1)/(2s)$ this proposition appears in \cite{FSa}. There it is also shown that the trace ideal index $r$ is smallest possible if $N\geq 3$ or if $N=2$ and $q\geq 4/3$. We note that the technique from \cite{Fr3} allows one also to obtain inequalities for $q>(N+1)/(2s)$.

\begin{proof}
We distinguish the following two cases,
\begin{align*}
\text{(A)} & \qquad (N-1)/2\leq s\leq (N+1)/2 \quad \text{and}\quad q\leq 2N/(N-1+2s) \quad\text{and}\quad q<2 \,,\\
\text{(B)} & \qquad \text{either}\ s<(N-1)/2 \\
& \qquad \text{or}\ (N-1)/2\leq s\leq (N+1)/2 \quad \text{and}\quad q> 2N/(N-1+2s) \,. 
\end{align*}
Note that case (A) corresponds to $r=2$ and case (B) to $r=(N-1)q/(N-qs)$.

\emph{Case} (A). We know that $(-\Delta-z)^{-s}$ is an integral operator with integral kernel
$$
\int_{\R^N} \frac{e^{i\xi\cdot(x-y)}}{(\xi^2-z)^s}\frac{d\xi}{(2\pi)^N} = \frac{2^{1-s}}{(2\pi)^{N/2}\, \Gamma(s)} \left( \frac{\sqrt{-z}}{|x-y|}\right)^{(N-2s)/2} K_{(N-2s)/2}(\sqrt{-z}|x-y|) \,,
$$
where we choose the branch of the square root on $\C\setminus(-\infty,0]$ with positive real part; see, e.g., \cite[Section III.2.8]{GS}. Bounds on Bessel functions (we give references for more precise bounds when dealing with case (B)) show that the absolute value of this kernel is bounded by
$$
C_{N,\rho,s} |z|^{(N-2s)/2} \left(\sqrt{|z|}|x-y|\right)^{-\rho} \,,
$$
where
$$
\begin{cases}
0\leq\rho\leq (N+1-2s)/2 & \text{if}\ N/2<s\leq (N+1)/2 \,,\\
0<\rho\leq (N+1-2s)/2 & \text{if}\ N/2 = s \,,\\
N-2s\leq\rho\leq (N+1-2s)/2 & \text{if}\ (N-1)/2\leq s<N/2 \,.
\end{cases}
$$
If $0\leq2\rho<N$ we can use the Hardy--Littlewood--Sobolev inequality to bound the Hilbert--Schmidt norm and obtain
$$
\left\| \alpha_1 (-\Delta-z)^{-s} \alpha_2 \right\|_2 \leq C_{N,\rho,s}' |z|^{(N-2s-\rho)/2} \|\alpha_1\|_{2N/(N-\rho)} \|\alpha_2\|_{2N/(N-\rho)} \,.
$$
Substituting $q=N/(N-\rho)$, we see that the assumptions on $q$ in case (A) correspond to the assumptions on $\rho$ and we obtain the claimed bounds.

\emph{Case} (B). We use complex interpolation similarly as in the proof of Theorem \ref{unifsob}. Since multiplication by $\alpha_j/|\alpha_j|$ is a bounded operator, we may assume that $\alpha_j\geq 0$ for $j=1,2$. We consider the same family $T_\zeta(z)$ of operators as in the proof of Theorem \ref{unifsob}. Bound \eqref{eq:222} implies immediately that
$$
\left\| \alpha_1^{\zeta/s} T_\zeta(z) \alpha_2^{\zeta/s} \right\| \leq A
\qquad\text{if}\ \re\zeta= 0
$$
with a constant $A$ depending only on $N$. On the other hand, the explicit form of the integral kernel of $T_\zeta(z)$ and the bounds in \cite[Proof of Thm. 2.3]{KRS} (see also \cite{CS}, \cite{FSa}) show that this integral kernel satisfies
$$
\left| T_\zeta(z)(x,y) \right| \leq B_{\re\zeta} |z|^{(N-1-2\re\zeta)/4}  |x-y|^{-(N+1-2\re\zeta)/2}
\qquad\text{if}\ 0<|\re\zeta- N/2|\leq 1/2 \,.
$$
Thus, if we assume in addition that $\re\zeta>1/2$, we can bound the Hilbert--Schmidt norm as before by the Hardy--Littlewood--Sobolev inequality and get
\begin{align*}
\left\| \alpha_1^{\zeta/s} T_\zeta(z) \alpha_2^{\zeta/s} \right\|_2 \leq B_{\re\zeta}' |z|^{(N-1-2\re\zeta)/4}
\|\alpha_1\|_{4N\re\zeta/(s(N-1+2\re\zeta))}^{\re\zeta/s} \|\alpha_2\|_{4N\re\zeta/(s(N-1+2\re\zeta))}^{\re\zeta/s}.
\end{align*}

We choose $t> s$ with $0<|t-N/2|\leq 1/2$ and $t>1/2$ and use complex interpolation with the lines $\re\zeta=0$ and $\re\zeta=t$ to get
$$
\left\| \alpha_1 T_s(z) \alpha_2  \right\|_{2t/s} \leq A^{(t-s)/t} \left(B_t'\right)^{s/t} |z|^{s(N-1-2t)/(4t)} \|\alpha_1\|_{4Nt/(s(N-1+2t))} \|\alpha_2\|_{4Nt/(s(N-1+2t))} \,.
$$
Substituting $t=(N-1)qs/(2(N-qs))$, we see that the assumptions on $q$ in case (B), plus the assumption $q\neq N^2/(2(2N-1))$, correspond to the assumptions on $t$.

Finally, let $q=N^2/((2N-1)s)$. In this case we use the same family $\tilde T_\zeta(z)$ of operators as in the proof of Theorem \ref{unifsob} and note that
$$
\left| \tilde T_\zeta(z)(x,y) \right| \leq \tilde B_{N/2} |z|^{-1/4}  |x-y|^{-1/2}
\qquad\text{if}\ \re\zeta= N/2 \,.
$$
As before, we can use the Hardy--Littlewood--Sobolev inequality to bound the Hilbert--Schmidt norm of $\alpha_1^{\zeta/s} T_\zeta(z) \alpha_2^{\zeta/s}$ for $\re\zeta=N/2$, and then we can deduce the claimed bound by complex interpolation. This proves the claimed bound in case (B).
\end{proof}

Combining this proposition with Lemma \ref{dn} we get

\begin{corollary}\label{bsti}
Let $0<\gamma< 1/2$ if $d=2$, $0<\gamma\leq1/2$ if $d=3$ and $0\leq\gamma\leq 1/2$ if $d\geq 4$. Then there is a constant $C_{\gamma,d}$ such that for all $\alpha_1,\alpha_2\in L^{2(2\gamma+d-1)}(\R^{d-1})$,
$$
\left\| \alpha_1 \Gamma (-\Delta-z)^{-1} \Gamma^* \alpha_2 \right\|_r \leq C_{\gamma,d} |z|^{-\gamma/(2\gamma+d-1)} \|\alpha_1\|_{2(2\gamma+d-1)} \|\alpha_2\|_{2(2\gamma+d-1)} \,.
$$
where $r=2$ if $d=2$ and $r=2(d-2)(2\gamma+d-1)/(d-1-2\gamma)$ if $d\geq 3$.
\end{corollary}

%%%%%%%%%%%%%%%%%%%%%%%%%%%%%%%%%%%%%%%%%%

\subsection*{Proof of Theorem \ref{main4}}

According to \cite[Prop.\!\! 4.1]{Fr3} (which generalizes a result in \cite{LSu}) the eigenvalues $(E_j)$ of $-\Delta+\sigma(x')\delta(x_d)$ in $L^2(\R^d)$ coincide with the eigenvalues $(E_j)$ of finite type of the analytic family $1+K$ in $\C\setminus[0,\infty)$, repeated according to algebraic multiplicity. Here $K$ denotes the Birman--Schwinger operator from \eqref{eq:bs}. Corollary \ref{bsti} combined with \cite[Thm.\!\! 3.1]{Fr3} (which is essentially from \cite{FSa} and relies on \cite{BGK}) yields
$$
\sum_j \delta(z_j) |z_j|^{-\frac12\left( 1 - \left( \frac{2\gamma r}{2\gamma+d-1} - 1 + \epsilon\right)_+ \right) } \leq C_{\gamma,d,\epsilon} \|\sigma\|_{2\gamma+d-1}^{(2\gamma+d-1)\left( 1+\left( \frac{2\gamma r}{2\gamma+d-1} - 1 + \epsilon\right)_+\right)/(2\gamma)}
$$
for any $\epsilon>0$ with $r=2$ if $d=2$ and $r=2(d-2)(2\gamma+d-1)/(d-1-2\gamma)$ if $d \geq 3$. Setting $\tau=\left( \frac{2\gamma r}{2\gamma+d-1} - 1 + \epsilon\right)_+$, we obtain the inequality in the theorem.

%%%%%%%%%%%%%%%%%%%%%%%%%%%%%%%%%%%%%%%%%%%%%%%%%%%%%%%%%%%%%%%%%%%%%%%%%

\subsection*{Acknowledgement}

Support through NSF grant DMS-1363432 is acknowledged.

%%%%%%%%%%%%%%%%%%%%%%%%%%%%%%%%%%%%%%%%%%%%%%%%%%%%%%%%%%%%%%%%%%%%%%%%%
%%%%%%%%%%%%%%%%%%%%%%%%%%%%%%%%%%%%%%%%%%%%%%%%%%%%%%%%%%%%%%%%%%%%%%%%%

\bibliographystyle{amsalpha}

\begin{thebibliography}{11}

\bibitem{AAD} A. A. Abramov, A. Aslanyan, E. B. Davies, \textit{Bounds on complex eigenvalues and resonances}. J. Phys. A \textbf{34} (2001), 57--72.

\bibitem{AF} R. A. Adams, J. J. F. Fournier, \textit{Sobolev spaces}. Second edition. Pure and Applied Mathematics \textbf{140}, Academic Press, Amsterdam, 2003.

\bibitem{BGK} A. Borichev, L. Golinskii, S. Kupin, \textit{A Blaschke-type condition and its application to complex Jacobi matrices}. Bull. London Math. Soc. \textbf{41} (2009), 117--123.

\bibitem{CS} S. Chanillo, E. Sawyer, \textit{Unique continuation for $\Delta+v$ and the C. Fefferman-Phong class}. 
Trans. Amer. Math. Soc. \textbf{318} (1990), no. 1, 275--300.

\bibitem{DeKa} M. Demuth, G. Katriel, \textit{Eigenvalue inequalities in terms of Schatten norm bounds on differences of semigroups, and application to Schr\"odinger operators}. Ann. Henri Poincar\'e \textbf{9} (2008), no. 4, 817--834.

\bibitem{DeHaKa0} M. Demuth, M. Hansmann, G. Katriel, \textit{On the discrete spectrum of non-selfadjoint operators}. J. Funct. Anal. \textbf{257} (2009), no. 9, 2742--2759.

\bibitem{DeHaKa} M. Demuth, M. Hansmann, G. Katriel, \textit{Eigenvalues of non-selfadjoint operators: A comparison of two approaches}, in: Mathematical Physics, Spectral Theory and Stochastic Analysis, Springer, 2013, 107--163.

\bibitem{E} P.  Exner, \textit{Leaky quantum graphs: a review}. In: Analysis on graphs and its applications, 523--564, Proc. Sympos. Pure Math., \textbf{77}, Amer. Math. Soc., Providence, RI, 2008.

\bibitem{F} R. L. Frank, \textit{Eigenvalue bounds for Schr\"odinger operators with complex potentials}. Bull. Lond. Math. Soc. \textbf{43} (2011), no. 4, 745--750.

\bibitem{Fr3} R. L. Frank, \textit{Eigenvalue bounds for Schr\"odinger operators with complex potentials. III}. Preprint (2015), http://arxiv.org/pdf/1510.03411v1.pdf

\bibitem{FL} R. L. Frank, A. Laptev, \textit{Spectral inequalities for Schr\"odinger operators with surface potentials}. In: Spectral theory of differential operators, 91--102, Amer. Math. Soc. Transl. Ser. 2, \textbf{225}, Amer. Math. Soc., Providence, RI, 2008. 

\bibitem{FLLS} R. L. Frank, A. Laptev, E. H. Lieb, R. Seiringer, \textit{Lieb--Thirring inequalities for Schr\"odinger operators with complex-valued potentials}. Lett. Math. Phys. \textbf{77} (2006), no. 3, 309--316.

\bibitem{FSa} R. L. Frank, J. Sabin, \textit{Restriction theorems for orthonormal functions, Strichartz inequalities and uniform Sobolev estimates}. Amer. J. Math., to appear.

\bibitem{FSh} R. L. Frank, R. G. Shterenberg, \textit{On the scattering theory of the Laplacian with a periodic boundary condition. II. Additional channels of scattering}. Doc. Math. \textbf{9} (2004), 57--77.

\bibitem{FS} R. L. Frank, B. Simon, \textit{Eigenvalue bounds for Schr\"odinger operators with complex potentials. III}. J. Spectr. Theory, to appear.

\bibitem{GS} I. M. Gel'fand, G. E. Shilov, \textit{Generalized functions. Vol. 1. Properties and operations}. Academic Press [Harcourt Brace Jovanovich, Publishers], New York-London, 1964 [1977].

\bibitem{Ha} M. Hansmann, \textit{An eigenvalue estimate and its application to non-self-adjoint Jacobi and Schr\"odinger operators}. Lett. Math. Phys. \textbf{98} (2011), no. 1, 79--95.

\bibitem{KRS} C. E. Kenig, A. Ruiz, C. D. Sogge, \textit{Uniform Sobolev inequalities and unique continuation for second order constant coefficient differential operators}. Duke Math. J. \textbf{55} (1987), no. 2, 329--347.

\bibitem{LS} A. Laptev, O. Safronov, \textit{Eigenvalue estimates for Schr\"odinger operators with complex potentials}. Comm. Math. Phys. \textbf{292} (2009), 29--54.

\bibitem{LT} E. H. Lieb, W. Thirring, \textit{Inequalities for the moments of the eigenvalues of the Schr\"odinger Hamiltonian and their relation to Sobolev inequalities}. Studies in Mathematical Physics. Princeton University Press, Princeton (1976), pp. 269--303.

\bibitem{LSu} Y. Latushkin, A. Sukhtayev, \textit{The algebraic multiplicity of eigenvalues and the Evans function revisited}. Math. Model. Nat. Phenom. \textbf{5} (2010), no. 4, 269--292.

\end{thebibliography}

\end{document}